\documentclass[11pt]{amsart}
\usepackage{extarrows}
\usepackage[colorlinks, citecolor=blue, dvipdfm, pagebackref]{hyperref}

\newtheorem{theorem}{Theorem}[section]
\newtheorem{lemma}[theorem]{Lemma}

\theoremstyle{definition}
\newtheorem{definition}[theorem]{Definition}

\newtheorem{example}[theorem]{Example}

\newtheorem{proposition}[theorem]{Proposition}
\newtheorem{corollary}[theorem]{Corollary}
\newtheorem{remark}[theorem]{Remark}
\newtheorem{conjecture}[theorem]{Conjecture}

\theoremstyle{remark}

\newcommand{\be}{\begin{equation}}
\newcommand{\ee}{\end{equation}}
\numberwithin{equation}{section}

\begin{document}

\title{The complex genera, symmetric functions and multiple zeta values}

\author{Ping Li}
\address{School of Mathematical Sciences, Fudan University, Shanghai 200433, China}

\email{pinglimath@fudan.edu.cn\\pinglimath@gmail.com}




\begin{abstract}
We examine the coefficients in front of Chern numbers for complex genera, and pay special attention to the $\text{Td}^{\frac{1}{2}}$-genus, the $\Gamma$-genus as well as the Todd genus. Some related geometric applications to hyper-K\"{a}hler and Calabi-Yau manifolds are discussed. Along this line and building on the work of Doubilet in 1970s, various Hoffman-type formulas for multiple-(star) zeta values and transition matrices among canonical bases of the ring of symmetric functions can be uniformly treated in a more general framework.
\end{abstract}

\maketitle


\section{Introduction}\label{introduction}
The notion of oriented genus was introduced by Hirzebruch (\cite{Hi}), which is a ring homomorphism from the rational oriented cobordism ring to $\mathbb{Q}$. Building on Thom's pioneer work (\cite{Th}), it turns out that oriented genera correspond one-to-one to monic formal power series and are rationally linear combination of Pontrjagin numbers. Two typical examples are Hirzebruch's $L$-genus and the $\hat{A}$-genus.
The Hirzebruch signature theorem implies that $L(\cdot)$ is the signature of the intersection pairing on $H^{2k}(M;\mathbb{R})$ and hence an integer (\cite[\S 8]{Hi}). The integrality of $\hat{A}(\cdot)$ for \emph{spin} manifolds was observed by Borel and Hirzebruch (\cite[\S 25]{BH}). This fact both motivated and was later explained by the Atiyah--Singer index theorem, which showed that the $\hat{A}$-genus of a spin manifold is the index of its Dirac operator (\cite[\S 5]{AS}). Since then various geometric and topological facets of the $L$-genus and $\hat{A}$-genus have been extensively investigated. Most notably, the former is related to the question of the homotopy invariance of the higher signatures, known as the Novikov conjecture, and the latter is deeply related to the existence of positive scalar curvature metrics on spin manifolds (\cite{Li},\cite{Hit},\cite{GL},\cite{Sto}).

In contrast to these, the \emph{arithmetic} properties of these two genera had not yet been fully studied, except the obvious connection with Bernoulli numbers (\cite[p.281]{MS}). In a recent work \cite{BB18}, Berglund and Bergstr\"{o}m showed that the coefficients in front of the Pontrjagin numbers of the $\hat{A}$-genus and $L$-genus can be expressed in terms of multiple-star zeta values and an alternating version respectively (\cite[Thms 1 and 4]{BB18}). In particular, all these coefficients are nonzero and their signs can be explicitly determined. Their main idea is to first express these coefficients in terms of values of the usual Riemann zeta function on even integers,
and then apply Hoffman-type formulas (\cite{Ho92}) to arrive at the desired results. Note that in this process some quite deep results in algebraic combinatorics are needed.

Inspired by the work \cite{BB18} and some ideas therein, \emph{the main purpose} of this article is to study the coefficients in front of Chern numbers for complex genera and then focus on three cases: the $\text{Td}^{\frac{1}{2}}$-genus, the $\Gamma$-genus as well as the Todd genus. In order to state some results in this article, let us introduce more notation.

An \emph{integer partition} $\lambda=(\lambda_1,\lambda_2,\ldots,\lambda_l)$ is a finite sequence of positive integers in non-increasing order: $\lambda_1\geq\lambda_2\geq\cdots\geq\lambda_l\geq 1$. Denote by $l(\lambda):=l$ and $|\lambda|:=\sum_{i=1}^l\lambda_i$ and they are called respectively the \emph{length} and \emph{weight} of the partition $\lambda$. These $\lambda_i$ are called \emph{parts} of $\lambda$. We may write $\lambda=(\lambda_1,\lambda_2,\ldots,\lambda_{l(\lambda)})$. It is also convenient to use another notation which indicates the number of times each integer appears: $\lambda=(1^{m_1(\lambda)}2^{m_2(\lambda)}\ldots)$. This means that $i$ appears with multiplicity $m_i(\lambda)$ among the parts $\lambda_i$. We use the notation $2\lambda$ to denote the partition $(2\lambda_1,2\lambda_2,\ldots,2\lambda_{l(\lambda)})$.

Recall that complex genera are defined via the complex cobordism ring, correspond one-to-one to monic formal power series and are rationally linear combination of Chern numbers (\cite{Mi}).
Let $\varphi$ be a complex genus whose associated power series is $Q(x)$. Denote by $e_i$ the $i$-th elementary symmetric function in the countably many (commuting) variables $x_1,x_2,\ldots$ (more details can be found in Section \ref{symmetric function}), and
\be\label{formal decomposition1}1+\sum_{i=1}^{\infty}Q_i(e_1,\cdots,e_i):=\prod_{i=1}^{\infty}Q(x_i),\ee
where $Q_i(e_1,\cdots,e_i)$ denotes the homogeneous part of degree $i$ in $\prod_{i=1}^{\infty}Q(x_i)$ \big($\text{deg}(x_i)=1$\big). Let $M$ be a compact, almost-complex manifold of real dimension $2n$ with Chern classes $c_i$. Then the value $\varphi(M)$ is given by
\be\label{formal decomposition2}\varphi(M)=\int_MQ_n(c_1,\cdots,c_n)=:\sum_{|\lambda|= n}b_{\lambda}(\varphi)C_{\lambda}[M],\ee
where $C_{\lambda}[M]$ is the Chern number associated to the partition $\lambda$ and $b_{\lambda}(\varphi)$ the coefficient in front of it.

For real numbers $t_1,\ldots,t_r>1$, define the two series
\be\label{MZV}\zeta(t_1,\ldots,t_r):=
\sum_{n_1>n_2>\cdots>n_r\geq 1}
\frac{1}{n_1^{t_1}n_2^{t_2}\cdots n_r^{t_r}}\ee
and
\be\label{MSZV}\zeta^{\star}(t_1,\ldots,t_r):=\sum_{n_1\geq n_2\geq\cdots\geq n_r\geq 1}\frac{1}{n_1^{t_1}n_2^{t_2}\cdots n_r^{t_r}},\ee
which were introduced by Hoffman (\cite{Ho92}) and Zagier (\cite{Za94}) independently and can be viewed as multiple versions of the classical Riemann zeta function
\be\label{R-Z functions}\zeta(t):=\sum_{n=1}^{\infty}\frac{1}{n^t},\qquad t>1.\ee
When these $t_1,\ldots,t_r\in\mathbb{Z}_{>0}$, (\ref{MZV}) and (\ref{MSZV}) are called \emph{multiple zeta values} (MZV for short) and \emph{multiple-star zeta values} (MSZV for short) respectively. We refer to \cite{Zh} for a thorough treatment on various algebraic facets of MZV and MSZV.

The symmetrization of $\zeta(t_1,\ldots,t_r)$ and $\zeta^{\star}(t_1,\ldots,t_r)$ are defined respectively by
\be\label{symmetrization of MZV}\zeta_S(t_1,\ldots,t_r):=\sum_{\sigma\in S_{r}}\zeta(t_{\sigma(1)},\ldots,t_{\sigma(r)})\ee
and
\be\label{symmetrization of MSZV}\zeta_S^{\star}(t_1,\ldots,t_r):=\sum_{\sigma\in S_{r}}\zeta^{\star}(t_{\sigma(1)},\ldots,t_{\sigma(r)}),\ee
where $S_r$ is the permutation group on $\{1,\ldots,r\}$.

The \emph{$\text{Td}^{\frac{1}{2}}$-genus} is the complex genus whose formal power series is $$Q(x)=(\frac{x}{1-e^{-x}})^{\frac{1}{2}},$$
the square root of the usual Todd genus.
Our first main result is the following
\begin{theorem}[=Corollary \ref{coro td1-2}]\label{td1-2}
The coefficients $b_{2\lambda}(\text{Td}^{\frac{1}{2}})$ of the Chern numbers $C_{2\lambda}(\cdot)$ in the $\text{Td}^{\frac{1}{2}}$-genus are given by
$$b_{2\lambda}(\text{Td}^{\frac{1}{2}})=
\frac{(-1)^{|\lambda|-l(\lambda)}}{(2\pi)^{2|\lambda|}\prod_im_i(\lambda)!}
\cdot\zeta^{\star}_S(2\lambda_{1},\ldots,2\lambda_{l(\lambda)})\qquad\big(0!:=1\big).$$
In particular, the coefficient $b_{2\lambda}(\text{Td}^{\frac{1}{2}})$ is nonzero for every partition $\lambda$. It is positive if $|\lambda|-l(\lambda)$ is even and negative if $|\lambda|-l(\lambda)$ is odd.
\end{theorem}
The $\text{Td}^{\frac{1}{2}}$-genus is of geometric importance in hyper-K\"{a}her geometry. An irreducible hyper-K\"{a}her manifold $M$ is a simply-connected compact K\"{a}hler manifold and $H^0(M,\Omega^2_M)$ is spanned by an everywhere non-degenerate holomorphic $2$-form $\tau$, whose complex dimension is necessarily even, say $2n$.  This definition is equivalent to the fact that, as a Riemannian manifold, its holonomy is equal to $\text{Sp}(n)$. Such a $\tau$ yields an isomorphism between the holomorphic tangent and cotangent bundles of $M$ and thus odd Chern classes are torsion elements. This means that only Chern numbers of the forms $C_{2\lambda}[M]$ ($|\lambda|=n$) are involved. We refer the reader to \cite{GHJ} for more details on irreducible hyper-K\"{a}hler manifolds. It is commonly believed that Chern numbers of irreducible hyper-K\"{a}hler manifolds should satisfy strong arithmetic constraints (see Section \ref{remark} for more details). We refer the reader to \cite[\S 4]{OSV22} and the references therein for some related open questions and comments. Hitchin and Sawon showed that $\text{Td}^{\frac{1}{2}}(M)>0$ (\cite{HS01}) and Jiang recently showed that $\text{Td}^{\frac{1}{2}}(M)<1$ when $n\geq 2$ (\cite[Cor.~5.1]{Ji23}). Applying Theorem \ref{td1-2}, we can reformulate Hitchin--Sawon and Jiang's results in the following version related to MSZV.
\begin{corollary}
Let $M$ be an irreducible hyper-K\"{a}hler manifold of complex dimension $2n\geq4$. Then we have
$$0<\sum_{\overset{l(\lambda)-n}{\text{are even}}}
\frac{\zeta^{\star}_S(2\lambda_{1},\ldots,2\lambda_{l(\lambda)})}{\prod_im_i(\lambda)!}
C_{2\lambda}[M]-\sum_{\overset{l(\lambda)-n}{\text{are odd}}}
\frac{\zeta^{\star}_S(2\lambda_{1},\ldots,2\lambda_{l(\lambda)})}{\prod_im_i(\lambda)!}
C_{2\lambda}[M]<(2\pi)^{2n}.$$
\end{corollary}

The \emph{$\Gamma$-genus} is the complex genus whose formal power series is
\be\label{Gamma function}Q(x)=\frac{1}{{\Gamma(1+x)}}:=\exp(\gamma x)\prod_{i= 1}^{\infty}\big[(1+\frac{x}{i})\exp(-\frac{x}{i})\big],\ee
where
\be\label{gamma constant}\gamma:=\lim_{n\rightarrow\infty}\big[(\sum_{i=1}^n\frac{1}{i})-\ln n\big]\ee
is usually called the \emph{Euler constant}. This $\Gamma$-genus was introduced by Libgober in \cite{Li99} in connection with mirror symmetry. It turns out that (\cite[p.142-143]{Li99}) the Chern class polynomials $Q_i\big(c_1(X),\dots,c_i(X)\big)$ \big(in the notation of (\ref{formal decomposition1})\big) of certain Calabi-Yau manifolds $X$ are  related to the coefficients of the generalized hypergeometric series expansion of the period of a mirror of $X$. He also showed in \cite[p.143]{Li99} that $Q_1(c_1)=\gamma c_1$ and the coefficient of $c_i$ in $Q_i(c_1,\ldots,c_i)$ is $\zeta(i)$ when $i\geq2$.

Building on his work \cite{Ho97}, Hoffman presented in \cite{Ho02} a closed formula for the coefficients $b_{\lambda}(\Gamma)$ by introducing a ring homomorphism from the ring of symmetric functions to $\mathbb{Q}$ and explaining $b_{\lambda}(\Gamma)$ as images of monomial symmetric functions (more details can be found in Section \ref{The gamma-genus}). We shall extend Hoffman's observation to show the following
\begin{theorem}[$\Leftarrow$Theorem \ref{thm gamma genus}]\label{gamma genus}
When $m_1(\lambda)=0$, i.e., $\lambda_{l(\lambda)}\geq 2$, the coefficients $b_{\lambda}(\Gamma)$ of the Chern numbers $C_{\lambda}[\cdot]$ in the $\Gamma$-genus are given by
$$b_{\lambda}(\Gamma)=
\frac{\zeta_S(\lambda_1,\ldots,\lambda_{l(\lambda)})}{\prod_{i}m_i(\lambda)!},$$
which is always positive.
In particular, these determine all the coefficients for Calabi-Yau manifolds.
\end{theorem}
\begin{remark}
In this article, by a Calabi-Yau manifold we mean in the weak sense that its first Chern class is a torsion element.
\end{remark}

Among the complex genera, the Todd genus, whose associated power series is $x/(1-e^{-x}),$ is the most classical as it equals to the arithmetic genus of compact complex manifolds due to the Hirzebruch-Riemann-Roch theorem. But, as we shall see in Section \ref{The Todd-genus}, the arithmetic expressions of the coefficients in the Todd genus cannot be made as compact as those of the $\text{Td}^{\frac{1}{2}}$-genus and $\Gamma$-genus, as illustrated by Theorems \ref{td1-2} and \ref{gamma genus}. Nevertheless, some partial arithmetic information on them can still be obtained, which will be briefly summarized in Proposition \ref{proposition td}.

The rest of this article is organized as follows. We prepare some preliminaries in Section \ref{preliminaries} on the partially ordered set consisting of the partitions of a finite set and four classical bases of the ring of symmetric functions. In Section \ref{section-Doubilet formula} Doubilet's constructions and formulas in \cite{Do72} will be extended and applied to uniformly treat various Hoffman-type formulas. With the tools in Sections \ref{preliminaries} and \ref{section-Doubilet formula} in hand, Section \ref{proof} is devoted to the study of the coefficients in front of Chern numbers for the general complex genera, and the three aforementioned complex genera will be discussed in detail in Section \ref{three cases}, from which Theorems \ref{td1-2} and \ref{gamma genus} follow as direct consequences. Some more remarks on the Chern numbers of irreducible hyper-K\"{a}hler manifolds will be presented in the last section, Section \ref{remark}.

\section{Preliminaries}\label{preliminaries}
\subsection{Partitions of a finite set}
The materials in this subsection can be found in \cite[\S 3]{Sta97}.

A \emph{poset} (partially ordered set) $(P,\leq)$ is a set $P$, together with a binary relation ``$\leq$" on $P$ which is reflexive ($x\leq x$, $\forall~x\in P$),  antisymmetric ($x\leq y$ and $y\leq x$ imply $x=y$), and transitive ($x\leq y$ and $y\leq z$ imply $x\leq z$). We use the obvious notation $x<y$ to mean $x\leq y$ and $x\neq y$. The poset we shall deal with in this article is the case below.

\begin{definition}\label{poset}
\begin{enumerate}
\item
A \emph{partition of a finite set} $S$ is a collection of disjoint nonempty subsets of $S$ whose union is $S$. Denote by $\Pi(S)$ the set consisting of all partitions of $S$. If $\pi=\{\pi_1,\ldots,\pi_l\}\in\Pi(S)$, each $\pi_i$ is called a \emph{block} of $\pi$. Let $l(\pi)=l$ and call it the \emph{length} of $\pi$. As in \cite{Sta97} we use the convention that $[n]:=\{1,\ldots,n\}$ and $\Pi_n:=\Pi([n])$.

\item
Define $\pi\leq\rho$ in $\Pi(S)$ if every block of $\pi$ is contained in a block of $\rho$. Note that $(\Pi(S),\leq)$ is a poset by easily checking the above-mentioned three properties of ``$\leq$". Also note that $(\Pi(S),\leq)$ has a \emph{unique} minimal (resp. maximal) element, denoted by $\hat{0}$ (resp. $\hat{1}$), whose blocks are precisely one-element subsets in $S$, i.e., $l(\hat{0})=|S|$ (resp. which has only one block $S$, i.e., $l(\hat{1})=1$). Here as usual $|\cdot|$ denotes the cardinality of a set.

\item
For $\pi,\rho\in\Pi(S)$, their exists a \emph{unique} element in $\Pi(S)$ denoted by $\pi\wedge\rho$, such that it is a lower bound of $\pi$ and $\rho$ (i.e., $\pi\wedge\rho\leq\pi$ and $\pi\wedge\rho\leq\rho$) and every lower bound $\tau$ of $\pi$ and $\rho$ satisfies $\tau\leq\pi\wedge\rho$. We call $\pi\wedge\rho$ the \emph{greatest lower bound} of $\pi$ and $\rho$ (\cite[p.102]{Sta97}). Note that if $\pi=\{\pi_i\}$ and $\rho=\{\rho_j\}$, then $\pi\wedge\rho=\hat{0}$ if and only if $|\pi_i\cap\rho_j|\leq 1$ for all $i$ and $j$.

\item
Each $\pi=\{\pi_1,\ldots,\pi_l\}\in\Pi(S)$ is naturally associated to an integer partition, denoted by $\lambda({\pi})$, whose parts are $|\pi_1|,\ldots,|\pi_l|$ and which is called the \emph{type} of $\pi$.
\end{enumerate}
\end{definition}
\begin{remark}
The notion of a partition of a finite set defined here should not be confused with that of an integer partition introduced in the Introduction.
\end{remark}
Every poset $(P,\leq)$ can be associated with a \emph{M\"{o}bius functions} $\mu: P\times P\rightarrow\mathbb{Z}$ defined inductively by
\begin{eqnarray}
\mu(x,y):=\left\{\begin{array}{ll}
1,&\text{if $x=y$},\\

-\sum_{x\leq z<y}\mu(x,z),&\text{if $x<y$},\\

0,&\text{otherwise},\\
\end{array} \right.\nonumber
\end{eqnarray}
whose importance lies in the M\"{o}bius inversion formula (\cite[p.116]{Sta97}).
For a thorough treatment on its basic properties and examples we refer the reader to \cite[\S 3.6-3.10]{Sta97}. In the following lemma we only encode the M\"{o}bius function of $(\Pi(S),\leq)$ in the form we shall use (\cite[Example 3.10.4]{Sta97}).
\begin{lemma}\label{Mobius lemma}
Assume that $\pi=\{\pi_1,\ldots,\pi_{l(\pi)}\}\leq\rho=\{\rho_1,\ldots,\rho_{l(\rho)}\}$ in $\Pi(S)$, and
\begin{eqnarray}
\left\{ \begin{array}{ll}
\big(\pi\rightarrow\rho_i\big):=\big\{\pi_j~|~\pi_j\subset\rho_i\big\},\\
~\\
|\pi\rightarrow\rho_i|:=\text{the cardinality of $\big(\pi\rightarrow\rho_i\big)$}.\\
\end{array} \right.\nonumber
\end{eqnarray}
Then the M\"{o}bius function $\mu(\pi,\rho)$ is given by
\be\label{general Mobius function}
\mu(\pi,\rho)=(-1)^{l(\pi)-l(\rho)}\prod_{i=1}^{l(\rho)}
\big(|\pi\rightarrow\rho_i|-1\big)!.\ee
In particular, \be\label{special Mobius function}\mu(\hat{0},\rho)=(-1)^{|S|-l(\rho)}\prod_{i=1}^{l(\rho)}\big(|\rho_i|-1\big)!,\ee
and moreover,
\be\label{absolute value formula}
\sum_{\pi\leq\rho}|\mu(\hat{0},\pi)|=\lambda(\rho)!,\qquad
\lambda(\rho)!:=\prod_i|\rho_i|!.\ee
\end{lemma}

\subsection{Symmetric functions}\label{symmetric function}
The standard references of this subsection are \cite[\S 1]{Ma95} and \cite[\S 7]{Sta99}.

Let $\mathbb{Q}[[x_1,x_2,\ldots]]$ be the ring of formal power series over $\mathbb{Q}$ in a countably infinite set of (commuting) variables $x_i$. An $f(x_1,x_2,\ldots)\in \mathbb{Q}[[x_1,x_2,\ldots]]$ is called a \emph{symmetric function} if it satisfies
$$f(x_{\sigma(1)},x_{\sigma(2)},\ldots)=f(x_1,x_2,\ldots),
\qquad\forall~\sigma\in S_k,~\forall~k\in\mathbb{Z}_{>0}.$$
Here, if $\sigma\in S_k$, $\sigma(i)=i$ for $i>k$ is understood.

Let $\Lambda^k(\mathbf{x})$ be the vector space of symmetric functions of homogeneous degree $k$ \big($\text{deg}(x_i):=1$\big). Then the \emph{ring} of symmetric functions $\Lambda(\mathbf{x}):=\bigoplus_{n=0}^{\infty}\Lambda^k(\mathbf{x})$ consists of all symmetric functions with bounded degree.

In the rest of this subsection we assume that $\lambda=(\lambda_1,\ldots,\lambda_{l(\lambda)})$ is an integer partition of weight $n$.

The \emph{elementary} symmetric function $e_k\in\Lambda^k(\mathbf{x})$ is defined by
$$e_k=e_k(x_1,x_2,\ldots):=\sum_{1\leq i_1<i_2<\cdots<i_k}x_{i_1}x_{i_2}\cdots x_{i_k},$$
and
\be\label{e-def}e_{\lambda}(\mathbf{x}):=
\prod_{i=1}^{l(\lambda)}e_{\lambda_i}\in\Lambda^n(\mathbf{x}).\ee

The \emph{power sum} symmetric function $p_k\in\Lambda^k(\mathbf{x})$ is defined by
$$p_k=p_k(x_1,x_2,\ldots):=\sum_{i=1}^{\infty}x_i^k,$$
and
\be\label{p-def}p_{\lambda}(\mathbf{x}):=
\prod_{i=1}^{l(\lambda)}p_{\lambda_i}\in\Lambda^n(\mathbf{x}).\ee

The \emph{monomial} symmetric function $m_{\lambda}(\mathbf{x})\in\Lambda^n(\mathbf{x})$ is defined by
$$m_{\lambda}(\mathbf{x}):=
\sum_{(\alpha_1,\alpha_2,\ldots)}x_1^{\alpha_1}x_2^{\alpha_2}\cdots,$$
where the sum is over all \emph{distinct} permutations $(\alpha_1,\alpha_2,\ldots)$ of the entries of the vector $\lambda=(\lambda_1,\ldots,\lambda_{l(\lambda)},0,\ldots)$. In other words, $m_{\lambda}(\mathbf{x})$ is the \emph{smallest} symmetric function containing the monomial $x_1^{\lambda_1}x_2^{\lambda_2}\cdots x_{l(\lambda)}^{\lambda_{l(\lambda)}}.$

The \emph{complete} symmetric function $h_k\in\Lambda^k(\mathbf{x})$ is defined by
\be\label{complete function} h_k=h_k(x_1,x_2,\ldots):=\sum_{\overset{\text{integer partitions $\mu$,}}{|\mu|=k}}m_{\mu}(\mathbf{x}),\ee
and
$$h_{\lambda}(\mathbf{x}):=
\prod_{i=1}^{l(\lambda)}h_{\lambda_i}\in\Lambda^n(\mathbf{x}).$$

It is well-known that (\cite[\S 1.2]{Ma95}) the four sets \be\label{four bases}\big\{e_{\lambda}~\big|~|\lambda|=n\big\},\qquad \big\{p_{\lambda}~\big|~|\lambda|=n\big\},\qquad \big\{m_{\lambda}~\big|~|\lambda|=n\big\},\qquad \big\{h_{\lambda}~\big|~|\lambda|=n\big\}\ee
are all bases of the vector space $\Lambda^n(\mathbf{x})$.

The following fact will be used in the sequel (\cite[p.292]{Sta99}).
\begin{lemma}
Let $\{x_1,x_2,\ldots\}$ and $\{y_1,y_2,\ldots\}$ be two countable sets of (commuting) variables $x_i$ and $y_j$. Then we have
\be\label{formula m-e relation}
\begin{split}
\prod_{i,j=1}^{\infty}(1+x_iy_j)&=1+\sum_{|\lambda|\geq 1}m_{\lambda}(\mathbf{x})
e_{\lambda}(\mathbf{y})\\
&=1+\sum_{|\lambda|\geq 1}m_{\lambda}(\mathbf{y})
e_{\lambda}(\mathbf{x}),
\end{split}
\ee
where the sum is over all positive integer partitions.
\end{lemma}

\section{Doubilet's formulas and applications}\label{section-Doubilet formula}
Since the vector space $\Lambda^n(\mathbf{x})$ has four bases in (\ref{four bases}), a natural question is what the transition matrices are between these four bases. We denote by, for instance, $M(e,m)$ the transition matrix $(M_{\lambda\mu})$ of coefficients in the equations $$e_{\lambda}(\mathbf{x})=\sum_{\mu}M_{\lambda\mu}\cdot m_{\mu}(\mathbf{x}),$$
and other transition matrices are similarly denoted. Except the three cases $M(e,m)$, $M(h,m)$ and $M(p,m)$, whose entries are quite easy to describe by their very definitions (\cite[\S7.4-7.7]{Sta99}), the entries in other transition matrices are not so direct to describe. In \cite{Do72} Doubilet applied the M\"{o}bius inversion to give a unified and compact treatment on the entries of \emph{all} these transition matrices.

In this section, we shall explain that Doubilet's constructions can be extended to a more general framework into which two useful situations fit well. One will yield the transition matrices of the four bases (\ref{four bases}), as originally considered by Doubilet in \cite{Do72}. The other will lead to various Hoffman-type formulas (\cite{Ho92}, \cite{Ho19}, \cite{Zh}) which relate the symmetrization of multiple-(star) zeta functions (\ref{symmetrization of MZV}), (\ref{symmetrization of MSZV}) and various variants to the original Riemann zeta function (\ref{R-Z functions}).

The following definition is an extension to \cite[\S 3]{Do72}. The form we adopt below is also inspired by the arguments in \cite[\S 3]{BB18}.
\begin{definition}\label{def of p,m,e,h}
\begin{enumerate}
\item
Fix a \emph{finite} set $S=\{a_1,\ldots,a_{|S|}\}$ and a (possibly infinite) set $D$. Let $x(a_i,n)$ be (commuting) variables parametrized by $a_i\in S$ and $n\in D$, and
$$x(T,n):=\prod_{a_i\in T}x(a_i,n),\qquad\forall~T\subset S,~\forall~n\in D.$$

\item
Let $a_i,a_j\in S$ and $\pi\in\Pi(S)$. The notation ``$a_i\overset{\pi}{\sim}a_j$" is used to denote that $a_i$ and $a_j$ belong to the same block of $\pi$.

\item
For an integer partition $\lambda=(\lambda_1,\ldots,\lambda_l)$, $\lambda!:=\prod_{i=1}^l\lambda_i!.$

\item
Let $\pi=\{\pi_1,\ldots,\pi_l\}\in\Pi(S)$ and define
\be\label{p}
p(\pi):=\sum_{n_1,\ldots,n_l\in D}x(\pi_1,n_1)x(\pi_2,n_2)\cdots x(\pi_l,n_l),\ee
\be\label{m}
m(\pi):=\sum_{\overset{n_1,\ldots,n_l\in D,}{\text{$n_i$ are distinct}}}x(\pi_1,n_1)x(\pi_2,n_2)\cdots x(\pi_l,n_l),\ee

\be\label{e}
e(\pi):=\sum_{\overset{n_1,\ldots,n_{|S|}\in D,}{\text{$n_i\neq n_j$ if $a_i\overset{\pi}{\sim}a_j$}}}x(a_1,n_1)x(a_2,n_2)\cdots x(a_{|S|},n_{|S|}),\ee
and
\be\label{h}h(\pi):=\sum_{\rho\in\Pi(S)}\lambda(\pi\wedge\rho)!\cdot m(\rho),\ee
where recall from Definition \ref{poset} that $\pi\wedge\rho$ is the greatest lower bound of $\pi$ and $\rho$, and $\lambda(\cdot)$ is the type of a partition in $\Pi(S)$ introduced in Definition \ref{poset}.
\end{enumerate}
\end{definition}

When taking $D=\mathbb{Z}_{>0}$ and $x(a_i,n)=x_n$ independent of $a_i\in S$, Definition \ref{def of p,m,e,h} specializes to the following example, which is exactly the case considered in \cite[Thms 1 and 5]{Do72} and justifies the notation in (\ref{p})-(\ref{h}).
\begin{example}[Doubilet]\label{Doubilet example}
Take $D=\mathbb{Z}_{>0}$ and $x(a_i,n)=x_n$ in Definition \ref{def of p,m,e,h}. Then $x(\pi_i,n_i)=(x_{n_i})^{|\pi_i|}$ and  (\ref{p})-(\ref{h}) become
\begin{eqnarray}
\left\{\begin{array}{ll}
p(\pi)=p_{\lambda(\pi)}(\mathbf{x}),\\
~\\
m(\pi)=\Big[\prod_im_i\big(\lambda(\pi)\big)!\Big]\cdot
m_{\lambda(\pi)}(\mathbf{x}),\\
~\\
e(\pi)=\lambda(\pi)!\cdot e_{\lambda(\pi)}(\mathbf{x}),\\
~\\
h(\pi)=\lambda(\pi)!\cdot h_{\lambda(\pi)}(\mathbf{x}).\\
\end{array} \right.\nonumber
\end{eqnarray}
\end{example}
\begin{remark}
Since the constructions and symbols used in \cite{Do72} are different from ours, we briefly explain that Example \ref{Doubilet example} is \emph{exactly} the case treated in \cite{Do72} for the reader's convenience. Following the notation in \cite[p.379]{Do72}, let $$F:=\{f:~S\rightarrow\mathbb{Z}_{>0}\},\qquad \gamma(f):=\prod_{i\in\mathbb{Z}_{>0}}(x_i)^{|f^{-1}(i)|},\qquad \ker(f):=\{f^{-1}(i)~|~i\in\mathbb{Z}_{>0}\}\in\Pi(S).$$
Then
$$p(\pi)=\sum_{n_1,\ldots,n_l\in \mathbb{Z}_{>0}}(x_{n_1})^{|\pi_1|}\cdots(x_{n_l})^{|\pi_l|}
=\sum_{\{f\in F|\ker(f)\geq\pi\}}\gamma(f)$$
and
$$m(\pi)=\sum_{\overset{n_1,\ldots,n_l\in\mathbb{Z}_{>0},}{\text{$n_i$ are distinct}}}(x_{n_1})^{|\pi_1|}\cdots(x_{n_l})^{|\pi_l|}
=\sum_{\{f\in F|\ker(f)=\pi\}}\gamma(f).$$
Recall from Definition \ref{poset} that $\ker(f)\wedge\pi=\hat{0}$ is equivalent to $|f^{-1}(i)\cap\pi_j|\leq 1$ for all $i$ and $j$. Thus
$$e(\pi)=\sum_{\overset{n_1,\ldots,n_{|S|}\in \mathbb{Z}_{>0},}{\text{$n_i\neq n_j$ if $a_i\overset{\pi}{\sim}a_j$}}}x_{n_1}x_{n_2}\cdots x_{n_{|S|}}=\sum_{\{f\in F|\ker(f)\cap\pi=\hat{0}\}}\gamma(f).$$
The original definition of $h(\pi)$ in \cite[\S 4]{Do72} is a little complicated, but it turns out that it can be expressed in terms of $m(\pi)$ like (\ref{h}) (\cite[Thm 6]{Do72}), which we take as its definition.
\end{remark}

Another example we are interested in is the following
\begin{example}\label{MZV example}
Let real numbers $t_1,\ldots,t_r>1$ be given. Take $S=[r]$ , $D=\mathbb{Z}_{>0}$ and $x(a,n)=\frac{1}{n^{t_a}}$ in Definition \ref{def of p,m,e,h}. For $\pi=\{\pi_1,\ldots,\pi_l\}\in\Pi_r$, denote by $t_{\pi_i}:=\sum_{j\in\pi_i}t_j$.

Recall the notation in (\ref{MZV})-(\ref{symmetrization of MSZV}). It is not difficult to check that in this case (\ref{p})-(\ref{h}) become
\begin{eqnarray}
\left\{\begin{array}{ll}
p(\pi)=\prod_{i=1}^l\zeta(t_{\pi_i}),\\
~\\
m(\pi)=\zeta_S(t_{\pi_1},\ldots,t_{\pi_l}),\\
~\\
e(\pi)=\prod_{i=1}^l\zeta_S(t_{u_{i,1}},\ldots,t_{u_{i,|\pi_i|}}),~
\big(\pi_i:=\{u_{i,1},\ldots,s_{u,|\pi_i|}\}\big)\\
~\\
h(\pi)=\prod_{i=1}^l\zeta_S^{\star}(t_{u_{i,1}},\ldots,t_{u_{i,|\pi_i|}}).\\
\end{array} \right.\nonumber
\end{eqnarray}
\end{example}
The following transition formulas were obtained in \cite[Thms 2 and 7]{Do72} by applying the M\"{o}bius inversion on the poset $\Pi(S)$ as well as some other tricks.
\begin{theorem}[Doubilet]\label{Doubilet formula}
Let the notation be as in Definition \ref{def of p,m,e,h}. Then we have
\be\label{m-p-transition}m(\pi)=\sum_{\pi\leq\rho}\mu(\pi,\rho)p(\rho),\ee
and
\be\label{h-p-transition}h(\pi)=\sum_{\rho\leq\pi}
\big|\mu(\hat{0},\rho)\big|p(\rho),\ee
where $\mu(\cdot,\cdot)$ is the M\"{o}bius function of the poset $\Pi(S)$ given in Lemma \ref{Mobius lemma}.
\end{theorem}
\begin{proof}
For completeness as well as for the reader's convenience, we copy the proof from \cite{Do72} verbatim.

The definitions (\ref{p}) and (\ref{m}) themselves imply that  \be\label{trivial}p(\pi)=\sum_{\pi\leq\rho}m(\rho),\ee
from which, together with the M\"{o}bius inversion (\cite[p.116]{Sta97}), (\ref{m-p-transition}) follows directly. This is the first part in \cite[Thm 2]{Do72}, which has also been shown in \cite[\S 3]{BB18} independently.

For (\ref{h-p-transition}), we have (see \cite[p.385]{Do72})
\be\begin{split}
h(\pi)&=\sum_{\rho}\lambda(\pi\wedge\rho)!\cdot m(\rho)\\
&=\sum_{\rho}\Big[\sum_{\tau\leq\pi\wedge\rho}
|\mu(\hat{0},\tau)|\Big]\cdot m(\rho)\qquad\big(\text{by (\ref{absolute value formula})}\big)\\
&=\sum_{\rho}\Big[\sum_{\tau\leq\pi,\tau\leq\rho}
|\mu(\hat{0},\tau)|\Big]\cdot m(\rho)\\
&=\sum_{\tau\leq\pi}|\mu(\hat{0},\tau)|\sum_{\rho\geq\tau}m(\rho)\\
&=\sum_{\tau\leq\pi}|\mu(\hat{0},\tau)|p(\tau),\qquad\big(\text{by (\ref{trivial})}\big)
\end{split}\nonumber\ee
which is exactly the second part in \cite[Thm 7]{Do72}.
\end{proof}

\begin{remark}
\begin{enumerate}
\item
Doubilet indeed obtained \emph{all} the transition matrices among the four basis (\ref{four bases}) in terms of the M\"{o}bius function $\mu(\cdot,\cdot)$ in \cite{Do72}. For our later purpose we only need and hence state the above two in Theorem \ref{Doubilet formula}. Doubilet's constructions in Example \ref{Doubilet example} and various transformation formulas among these four bases have been generalized in \cite{RS} to the associated algebra of symmetric functions in \emph{noncommuting} variables.

\item
As mentioned above, (\ref{m-p-transition}) has also been observed in \cite[\S 3]{BB18} and subsequently applied to obtain the transformation formula (\ref{mx-px-transition}) and the Hoffman-type formula (\ref{Hoffman formula MZV}) below. Our treatment in Definition \ref{def of p,m,e,h}
is indeed partially inspired by some arguments in \cite[\S 3]{BB18}. 
\end{enumerate}
\end{remark}
Applying (\ref{m-p-transition}) to Example \ref{Doubilet example}, together with the concrete expression of the M\"{o}bius function in (\ref{general Mobius function}), leads to
\begin{corollary}[Doubilet]
Let $\lambda=(\lambda_1,\ldots,\lambda_{l(\lambda)})=
(1^{m_1(\lambda)}2^{m_2(\lambda)}\cdots)$ be an integer partition and $\lambda_{\pi_i}:=\sum_{j\in\pi_i}\lambda_j$. Then we have
\be\label{mx-px-transition}m_{\lambda}(\mathbf{x})=
\frac{1}{\prod_im_i(\lambda)!}\sum_{\pi\in\Pi_{l(\lambda)}}
\Bigg\{(-1)^{l(\lambda)-l(\pi)}
\prod_{i=1}^{l(\pi)}\Big[\big(|\pi_i|-1\big)!\cdot
p_{\lambda_{\pi_i}}(\mathbf{x})\Big]\Bigg\}.\ee
\end{corollary}

Now applying (\ref{m-p-transition}) \big(resp. (\ref{h-p-transition})\big) to  Example \ref{MZV example} by taking $\pi=\hat{0}$ (resp. $\pi=\hat{1}$), together with the help of (\ref{special Mobius function}), yields the following formula (\ref{Hoffman formula MZV}) \big((resp. (\ref{Hoffman formula MSZV})\big) due to Hoffman (\cite[Thms 2.1 and 2.2]{Ho92}) (recall $\hat{0}$ and $\hat{1}$ in Definition \ref{poset}).
\begin{corollary}[Hoffman]
For any real $t_1,\ldots,t_r>1$, we have
\be\label{Hoffman formula MZV}
\zeta_S(t_1,\ldots,t_r)=\sum_{\pi\in\Pi_{r}}\Bigg\{
(-1)^{r-l(\pi)}\prod_{i=1}^{l(\pi)}
\Big[\big(|\pi_i|-1\big)!\cdot\zeta(t_{\pi_i})\Big]\Bigg\},\ee
and
\be\label{Hoffman formula MSZV}
\zeta_S^{\star}(t_1,\ldots,t_r)=\sum_{\pi\in\Pi_{r}}
\prod_{i=1}^{l(\pi)}\Big[\big(|\pi_i|-1\big)!\cdot\zeta(t_{\pi_i})\Big],
\ee
where $t_{\pi_i}:=\sum_{j\in\pi_i}t_j$.
\end{corollary}
\begin{remark}
These Hoffman-type formulas later have various variants. For example, there is an odd version (\cite{Ho19}), an alternating version or a finite version (\cite[\S 7-8]{Zh}) and so on. In all these cases, we can suitably choose the set $D$ in Definition \ref{def of p,m,e,h} and apply Theorem \ref{Doubilet formula} to arrive at the expected formulas.
\end{remark}

\section{The coefficients of Chern numbers in complex genera}
\label{proof}
With the materials in Sections \ref{preliminaries} and \ref{section-Doubilet formula} in hand, we can now examine the coefficients in front of Chern numbers for general complex genera in this section, and then focus on three cases (the $\text{Td}^{\frac{1}{2}}$-genus, $\Gamma$-genus as well as Todd-genus) in the next section.

Let $\varphi$ be a complex genus whose associated power series is
$$Q(x)=1+\sum_{i=1}^{\infty}a_ix^i,$$
and
\be\label{formal decomposition4}
1+\sum_{n=1}^{\infty}Q_n\big(e_1(\mathbf{x}),\cdots,e_n(\mathbf{x})\big):=\prod_{i=1}^{\infty}Q(x_i),\ee
where as before $e_n(\mathbf{x})$ is the $n$-th elementary symmetric function of the variables $x_1,x_2,\ldots,$ and $Q_n$ the homogeneous part of degree $n$. Denote by
\be\label{formal decomposition3}
Q_n\big(e_1(\mathbf{x}),\cdots,e_n(\mathbf{x})\big)=:
b_n(\varphi)e_n(\mathbf{x})+
\sum_{|\lambda|=n,~l(\lambda)\geq2}b_{\lambda}(\varphi)e_{\lambda}(\mathbf{x}).\ee
Then according to (\ref{formal decomposition1}) and (\ref{formal decomposition2}), for each real $2n$-dimensional compact almost-complex manifold $M$ and an integer partition $\lambda$ of weight $n$, the coefficient in front of the Chern number $C_{\lambda}[M]$ in the value $\varphi(M)$ is exactly $b_{\lambda}(\varphi)$. These $b_{\lambda}(\varphi)$ can be determined by the power series $Q(x)$ in the following manner.
\begin{lemma}\label{technical lemma}
Let $f(x)$ be the formal power series determined by
$$Q(x)=1+\sum_{i=1}^{\infty}a_ix^i=:\frac{x}{f(x)}.$$
\begin{enumerate}
\item
If the coefficients $a_i$ in $Q(x)$ are viewed as $a_i:=e_i(\mathbf{y})=e_i(y_1,y_2,\ldots),$
the $i$-th elementary symmetric functions of the variables $y_1,y_2,\ldots,$ then \be\label{b-m}b_{\lambda}(\varphi)=m_{\lambda}(\mathbf{y}).\ee

\item
The coefficients $b_n(\varphi)$ in (\ref{formal decomposition3}) are determined by
\be\label{cauchy}
1+\sum_{n=1}^{\infty}(-1)^n\cdot b_n(\varphi)\cdot x^n=x\cdot\frac{f'(x)}{f(x)}.
\ee
\end{enumerate}
\end{lemma}
\begin{remark}
Our notation $f(x)$ here is compatible with that used in the definition of \emph{elliptic genus} (\cite[p.17]{HBJ}).
\end{remark}
\begin{proof}
By (\ref{formal decomposition4}) and (\ref{formal decomposition3}) we have
\be\begin{split}
1+\sum_{|\lambda|\geq1}b_{\lambda}(\varphi)
e_{\lambda}(\mathbf{x})&=\prod_{i=1}^{\infty}Q(x_i)\\
&=\prod_{i=1}^{\infty}
\big(1+\sum_{k=1}^{\infty}a_kx_i^k\big)\\
&=\prod_{i=1}^{\infty}
\big[1+\sum_{k=1}^{\infty}e_k(\mathbf{y})x_i^k\big]\\
&=\prod_{i,j=1}^{\infty}\big(1+x_iy_j\big)\\
&\overset{(\ref{formula m-e relation})}{=}1+\sum_{|\lambda|\geq1}m_{\lambda}(\mathbf{y})e_{\lambda}(\mathbf{x}),\\
\end{split}\nonumber\ee
from which (\ref{b-m}) follows. The identity (\ref{cauchy}) is well-known (\cite[p.11]{Hi}). We provide a proof here for the sake of completeness.
\be\begin{split}
x\cdot\frac{f'(x)}{f(x)}&=1-x\frac{\mathrm{d}}{\mathrm{d}x}\log Q(x)\\
&=1-x\frac{\mathrm{d}}{\mathrm{d}x}
\log\prod_{i=1}^{\infty}(1+y_ix)\qquad\big(a_i:=e_i(\mathbf{y})\big)\\
&=1-x\sum_{i=1}^{\infty}\frac{y_i}{1+y_ix}\\
&=1+\sum_{n=1}^{\infty}(-1)^n(\sum_{i=1}^{\infty}y_i^n)x^n.
\end{split}\nonumber\ee
Thus $$\sum_{i=1}^{\infty}y_i^n=p_n(\mathbf{y})=m_{(n)}(\mathbf{y}),$$
which, together with (\ref{b-m}), leads to (\ref{cauchy}).
\end{proof}
Combining (\ref{b-m}), (\ref{mx-px-transition}) and the fact that $p_i(\mathbf{x})=m_{(i)}(\mathbf{x})$ yield the following closed formula for $b_{\lambda}(\varphi)$ in terms of those $b_n(\varphi)$ determined by (\ref{cauchy}).
\begin{lemma}
The coefficients $b_{\lambda}(\varphi)$ in front of the Chern numbers $C_{\lambda}[\cdot]$ for the genus $\varphi$ is given by
\be\label{formula for coefficient}
b_{\lambda}(\varphi)=\frac{1}{\prod_im_i(\lambda)!}\sum_{\pi\in\Pi_{l(\lambda)}}
\Bigg\{(-1)^{l(\lambda)-l(\pi)}\prod_{i=1}^{l(\pi)}\Big[\big(|\pi_i|-1\big)!\cdot
b_{\lambda_{\pi_i}}(\varphi)\Big]\Bigg\},\ee
where $\lambda_{\pi_i}=\sum_{j\in\pi_i}\lambda_j$ and $b_i(\varphi)$  are determined by the identity (\ref{cauchy}).
\end{lemma}
\begin{remark}
This has been observed in \cite[Thm 5]{BB18} for the case of oriented genera. However, the essence of both proofs is the same.
\end{remark}

\section{The three classical complex genera}\label{three cases}
Before treating the three cases in detail, we fix the following notation. The \emph{Bernoulli numbers} $B_{2n}$ ($n\geq 1$) are defined by
\be\label{Bernoulli}\frac{x}{e^x-1}=:1-\frac{1}{2}x+\sum_{n=1}^{\infty}
\frac{B_{2n}}{(2n)!}x^{2n},\ee
and it is well-known that (\cite[p.131]{HBJ})
\be\label{Bernoulli-Riemann-Zeta}\zeta(2n)=
\frac{(-1)^{n-1}(2\pi)^{2n}}{2\cdot(2n)!}B_{2n}.\ee

\subsection{The $\text{Td}^{\frac{1}{2}}$-genus}
\begin{example}\label{calculationtd1-2}
If $Q(x)=(\frac{x}{1-e^{-x}})^{\frac{1}{2}},$ then
the coefficients $b_n=b_n(\text{Td}^{\frac{1}{2}})$ are given by
\begin{eqnarray}\label{coefficient-td1-2}
\left\{\begin{array}{ll}
b_1=\frac{1}{4},\\

b_{2n+1}=0,&(n\geq 1)\\

b_{2n}=\frac{(-1)^{n-1}}{(2\pi)^{2n}}\zeta(2n).&(n\geq 1)\\
\end{array} \right.
\end{eqnarray}
\end{example}
\begin{proof}
In this case the associated $f(x)$ is $\big[x(1-e^{-x})\big]^{\frac{1}{2}}$. Direct calculations show that
\be\begin{split}
1+\sum_{n=1}^{\infty}(-1)^nb_n\cdot x^n&\overset{(\ref{cauchy})}{=}x\cdot\frac{f'(x)}{f(x)}\\
&=\frac{1}{2}(1+\frac{x}{e^x-1})\\
&\overset{(\ref{Bernoulli})}{=}1-\frac{1}{4}x+\sum_{n=1}^{\infty}
\frac{B_{2n}}{2\cdot(2n)!}x^{2n}\\
&\overset{(\ref{Bernoulli-Riemann-Zeta})}{=}1-\frac{1}{4}x+\sum_{n=1}^{\infty}
\frac{(-1)^{n-1}\zeta(2n)}{(2\pi)^{2n}}x^{2n},
\end{split}
\nonumber\ee
which yields (\ref{coefficient-td1-2}).
\end{proof}

Our main observation for the coefficients $b_{\lambda}(\text{Td}^{\frac{1}{2}})$ is
\begin{theorem}\label{theorem td1-2}
Assume that $m_1(\lambda)=0$. Then $b_{\lambda}(\text{Td}^{\frac{1}{2}})=0$ unless $|\lambda|$ is even. If $|\lambda|$ is even, we have
\be\label{coefficient m1=0}b_{\lambda}(\text{Td}^{\frac{1}{2}})=
\frac{(-1)^{l(\lambda)-\frac{|\lambda|}{2}}}
{(2\pi)^{|\lambda|}\cdot\prod_i m_i(\lambda)!}
\sum_{\overset{\pi\in\Pi_{l(\lambda)}}{\text{$\lambda_{\pi_i}$ are even}}}
\prod_{i=1}^{l(\pi)}\big[(|\pi_i|-1)!\cdot\zeta(\lambda_{\pi_i})\big].
\ee
\end{theorem}
\begin{proof}
If $|\lambda|$ is odd, then for each $\pi=\{\pi_1,\ldots,\pi_{l(\pi)}\}\in\Pi_{l(\lambda)}$ at least some $\lambda_{\pi_{i_0}}$ is odd since $|\lambda|=\sum_i\lambda_{\pi_{i}}$. As $m_1(\lambda)=0$, this implies that $\lambda_{\pi_{i_0}}\geq3$ and hence $b_{\lambda_{\pi_{i_0}}}=0$ due to (\ref{coefficient-td1-2}). This tells us that each summand on the right hand side (RHS for short) of (\ref{formula for coefficient}) is zero and consequently $b_{\lambda}(\text{Td}^{\frac{1}{2}})=0$.

If now $|\lambda|$ is even, the above analysis says that only those $\pi=\{\pi_1,\ldots,\pi_{l(\pi)}\}\in\Pi_{l(\lambda)}$ such that all $\lambda_{\pi_i}$ are even can contribute to the RHS of (\ref{formula for coefficient}). In this case the summand on the RHS of (\ref{formula for coefficient}) is
\be\begin{split}
&(-1)^{l(\lambda)-l(\pi)}\prod_{i=1}^{l(\pi)}\Big[\big(|\pi_i|-1\big)!\cdot
b_{\lambda_{\pi_i}}(\varphi)\Big]\\
\overset{(\ref{coefficient-td1-2})}{=}&(-1)^{l(\lambda)-l(\pi)}\prod_{i=1}^{l(\pi)}\Big[\big(|\pi_i|-1\big)!\cdot
\frac{(-1)^{\frac{\lambda_{\pi_i}}{2}-1}}{(2\pi)^{\lambda_{\pi_i}}}\zeta(\lambda_{\pi_i})\Big]\\
=&\frac{(-1)^{l(\lambda)-\frac{|\lambda|}{2}}}{(2\pi)^{|\lambda|}}
\prod_{i=1}^{l(\pi)}\Big[\big(|\pi_i|-1\big)!\cdot\zeta(\lambda_{\pi_i})\Big]
\end{split}\nonumber\ee
and leads to the desired (\ref{coefficient m1=0}).
\end{proof}

Together with (\ref{Hoffman formula MSZV}), a direct consequence of Theorem \ref{theorem td1-2} is a compact expression in terms of MSZV for those integer partitions whose parts are all even, which is exactly Theorem \ref{td1-2} in the Introduction.
\begin{corollary}[=Theorem \ref{td1-2}]\label{coro td1-2}
For the integer partitions $2\lambda=(2\lambda_1,\ldots,2\lambda_{l(\lambda)})$, we have
$$b_{2\lambda}(\text{Td}^{\frac{1}{2}})=
\frac{(-1)^{|\lambda|-l(\lambda)}}{(2\pi)^{2|\lambda|}\cdot\prod_im_i(\lambda)!}\cdot
\zeta^{\star}_S(2\lambda_{1},\ldots,2\lambda_{l(\lambda)}).$$
In particular each $b_{2\lambda}(\text{Td}^{\frac{1}{2}})$ is nonzero with sign $(-1)^{|\lambda|-l(\lambda)}$.
\end{corollary}

\subsection{The $\Gamma$-genus}\label{The gamma-genus}
\begin{definition}
Define an algebra homomorphism $T$ from the algebra of symmetric functions $\Lambda(\mathbf{x})$ to $\mathbb{R}$ by requiring the values of the power sum symmetric functions $p_i(\mathbf{x})$ as follows.
\be\label{values of p}
T:~\Lambda(\mathbf{x})\rightarrow\mathbb{R},\qquad
T\big(p_i(\mathbf{x})\big):=\left\{\begin{array}{ll}
\zeta(i),&i\geq2\\
\gamma,&i=1\\
1,&i=0\\
\end{array} \right.
\qquad\big(p_0(\mathbf{x}):=1\big)
\ee
where $\gamma$ is the Euler constant introduced in (\ref{gamma constant}). Note that the algebra homomorphism $T$ is completely determined by (\ref{values of p}) since $\big\{p_{\lambda}(\mathbf{x})\big\}$ is a basis of $\Lambda(\mathbf{x})$.
\end{definition}
The following result is due to Hoffman (\cite[p.972]{Ho02}) building on \cite[\S 1]{Li99}.
\begin{theorem}[Hoffman]
The coefficient $b_{\lambda}(\Gamma)$ in front of the Chern number $C_{\lambda}[\cdot]$ for the $\Gamma$-genus defined in (\ref{Gamma function}) is given by the image of $m_{\lambda}(\mathbf{x})$ under $T$: \be\label{Hoffman}b_{\lambda}(\Gamma)=T\big(m_{\lambda}(\mathbf{x})\big).\ee
\end{theorem}
\begin{proof}
Since the proof in \cite{Ho02} resorts to some materials in \cite[\S 5]{Ho97}, for the reader's convenience we provide a proof here. First we have (\cite[p.45]{Er53})
\be\label{1}\log\Gamma(1-x)=\gamma x+\sum_{i=2}^{\infty}\frac{\zeta(i)}{i}\cdot x^i,\ee
and the symmetric functions $e_i=e_i(\mathbf{x})$ and $p_i=p_i(\mathbf{x})$ are related by (\cite[p.21-23]{Ma95})
\be\label{2}1+\sum_{i=1}^{\infty}e_i\cdot t^i
=\exp\big[-\sum_{i=1}^{\infty}\frac{p_i}{i}\cdot(-t)^i\big].\ee
Thus
\be\label{3}
\begin{split}
1+\sum_{i=1}^{\infty}T(e_i)\cdot t^i&\overset{(\ref{2})}{=}
\exp\big[-\sum_{i=1}^{\infty}\frac{T(p_i)}{i}\cdot(-t)^i\big]\\
&\overset{(\ref{values of p})}{=}\exp\big[\gamma t-\sum_{i=2}^{\infty}\frac{\zeta(i)}{i}\cdot(-t)^i\big]\\
&\overset{(\ref{1})}{=}\frac{1}{\Gamma(1+t)}.\end{split}\ee
Thus
\be\begin{split}
1+\sum_{|\lambda|\geq1}b_{\lambda}(\Gamma)
e_{\lambda}(\mathbf{y})&=\prod_{i=1}^{\infty}\frac{1}{\Gamma(1+y_i)}\\
&\overset{(\ref{3})}{=}\prod_{i=1}^{\infty}
\big[1+\sum_{k=1}^{\infty}T(e_k)\cdot y^k_i\big]\\
&=T\big[\prod_{i,j=1}^{\infty}(1+y_ix_j)\big]\\
&\overset{(\ref{formula m-e relation})}{=}
1+\sum_{|\lambda|\geq1}T(m_{\lambda})e_{\lambda}(\mathbf{y}),
\end{split}\nonumber\ee
which completes the proof.
\end{proof}
With these materials in hand, we can proceed to show the result below, from which Theorem \ref{gamma genus} in the Introduction follows.
\begin{theorem}\label{thm gamma genus}
The coefficient $b_{\lambda}(\Gamma)$ in front of the Chern number $C_{\lambda}[\cdot]$ for the $\Gamma$-genus are given by
\be\label{identity}\begin{split} b_{\lambda}(\Gamma)&=\frac{1}{\prod_im_i(\lambda)!}\sum_{\pi\in\Pi_{l(\lambda)}}
\Bigg\{(-1)^{l(\lambda)-l(\pi)}
\prod_{i=1}^{l(\pi)}\Big[\big(|\pi_i|-1\big)!\cdot
\zeta(\lambda_{\pi_i})\Big]\Bigg\}\qquad\big(\zeta(1):=\gamma\big)\\
&\xlongequal{m_1(\lambda)=0}\frac{\zeta_S(\lambda_1,\ldots,\lambda_{l(\lambda)})}{\prod_{i}m_i(\lambda)!}.
\end{split}\ee
In the first identity $\zeta(\lambda_{\pi_i})=\gamma$ occurs if and only if $|\pi_i|=1$, say $\pi_i=\{j\}$, and $\lambda_j=1$. In particular, the second identity holds true for Calabi-Yau manifolds.
\end{theorem}
\begin{proof}
Due to (\ref{Hoffman}) and (\ref{values of p}), we apply the ring homomorphism $T(\cdot)$ on both sides of (\ref{mx-px-transition}) to yield the first identity in (\ref{identity}) with the convention that $\zeta(1):=\gamma$.

If moreover $m_1(\lambda)=0$, all these $\lambda_{\pi_i}\geq2$ and the second identity follows from the Hoffman-type formula (\ref{Hoffman formula MZV}).
\end{proof}

\subsection{The Todd-genus}\label{The Todd-genus}
Completely analogous to Example \ref{calculationtd1-2} and Theorem \ref{theorem td1-2} we can deduce that
\begin{example}
If $Q(x)=\frac{x}{1-e^{-x}},$ the coefficients $b_n=b_n(\text{Td}$) are
\begin{eqnarray}\label{coefficient-td}
\left\{\begin{array}{ll}
b_1=\frac{1}{2},\\

b_{2n+1}=0,&(n\geq 1)\\

b_{2n}=\frac{2\cdot(-1)^{n+1}}{(2\pi)^{2n}}\zeta(2n),&(n\geq 1)\\
\end{array} \right.
\end{eqnarray}
\end{example}
and that
\begin{proposition}\label{proposition td}
Assume that $m_1(\lambda)=0$. Then $b_{\lambda}(\text{Td})=0$ unless $|\lambda|$ is even. If $|\lambda|$ is even, we have
\be\label{coefficient td m1=0}b_{\lambda}(\text{Td})=
\frac{(-1)^{l(\lambda)-\frac{|\lambda|}{2}}}
{(2\pi)^{|\lambda|}\cdot\prod_i m_i(\lambda)!}
\sum_{\overset{\pi\in\Pi_{l(\lambda)}}{\text{$\lambda_{\pi_i}$ are even}}}
\Big\{2^{l(\pi)}\cdot\prod_{i=1}^{l(\pi)}\big[(|\pi_i|-1)!\cdot\zeta(\lambda_{\pi_i})\big]\Big\}.
\ee
In particular, in the cases of $|\lambda|$ be even and $m_1(\lambda)=0$, $b_{\lambda}(\text{Td})$ are nonzero with sign $(-1)^{l(\lambda)-\frac{|\lambda|}{2}}$.
\end{proposition}
\begin{remark}\label{SMZV remark}
\begin{enumerate}
\item
Unlike the formula (\ref{coefficient m1=0}) in Theorem \ref{theorem td1-2}, we are not able to deduce a compact formula in terms of MZV for those integer partitions whose parts are all even, as that in Corollary \ref{coro td1-2}, due to the appearance of the extra factor $2^{l(\pi)}$ in the RHS of (\ref{coefficient td m1=0}).

\item
An easy consequence of (\ref{coefficient td m1=0}) is that if $b_{\lambda}(\text{Td})\neq0$ for \emph{odd} $|\lambda|$, then $m_1(\lambda)\geq1$, i.e., the first Chern class $c_1$ is involved in the Chern number $C_{\lambda}[\cdot]$. This fact has been observed by Hirzebruch (\cite[p.14]{Hi}).

\item
Recently \emph{Schur multiple zeta values} (SMZV for short) were introduced and investigated in \cite{NPY18} by utilizing the semi-standard Young tableaux. They generalize MZV and MSZV in a natural way and also build a bridge to the theory of partitions. Related works can be found in \cite{Ba18}, \cite{BY18}, \cite{BC20}, and \cite{BKSYY23}. One of the referees of our paper pointed out that it would be interesting to see if the RHS of (\ref{coefficient td m1=0}) can be expressible in terms of SMZV. The author hopes to consider it in the near future.
\end{enumerate}
\end{remark}

\section{Some comments on Chern numbers of hyper-K\"{a}hler manifolds}\label{remark}
The known examples of irreducible hyper-K\"{a}hler manifolds are quite scarce. Up to deformations they are Hilbert schemes of points on K3 surfaces, generalized Kummer varieties (\cite{Be}) and two examples in dimensions $6$ and $10$ respectively (\cite{OG1},\cite{OG2}). Many positivity properties on these examples indicate that Chern numbers of irreducible hyper-K\"{a}hler manifolds should satisfy many constraints. The following conjecture was proposed in \cite[Appendix B]{Ni} and raised again recently in \cite[Question 4.8]{OSV22}.
\begin{conjecture}\label{conjecture Chern number}
All (monomial) Chern numbers of irreducible hyper-K\"{a}hler manifolds are positive.
\end{conjecture}
Another positivity conjecture on Chern character numbers was also raised in \cite{OSV22}. Before stating it, let us introduce one more notation. Let $x_1,\ldots,x_n$ be Chern roots of a compact almost-complex manifold $M$ of real dimension $2n$, i.e., the Chern classes $c_i(M)$ are viewed as $e_i(x_1,\ldots,x_n)$, the $i$-th elementary symmetric polynomials of $x_1,\ldots,x_n$. Let
$$\text{ch}_i(M):=\frac{x_1^i+\cdots+x_n^i}{i!}\in H^{2i}(M;\mathbb{Q}).$$
Then $\sum_{i=0}^n\text{ch}_i(M)$ is the usual Chern character of $M$. Given an integer partition $\lambda=(\lambda_1,\ldots,\lambda_{l(\lambda)})$ of weight $n$, the \emph{Chern character number} $\text{Ch}_{\lambda}[M]$ is defined by
$$\text{Ch}_{\lambda}[M]:=\int_M\prod_{i=1}^{l(\lambda)}\text{ch}_i(M)\in\mathbb{Q}.$$
By its definition any Chern character number is a rationally linear combination of Chern numbers and vice versa. It was shown in \cite[Prop.3.7]{OSV22} that all signed Chern character numbers $(-1)^n\text{Ch}_{2\lambda}$ ($|\lambda|=n$) of a $2n$-dimensional generalized Kummer variety are positive. Based on this, the following conjecture was raised (\cite[Question 4.7]{OSV22}).
\begin{conjecture}\label{conjecture Chern character number}
All signed Chern character numbers of irreducible hyper-K\"{a}hler manifolds are positive.
\end{conjecture}
The only \emph{known} positivity result on Chern numbers which holds true for \emph{all} irreducible hyper-K\"{a}hler manifolds seems to be the aforementioned $\text{Td}^{\frac12}[\cdot]>0$ due to Hitchin--Sawon (\cite{HS01}), to the best knowledge of the author.

One of the primary motivations of this article is to see the relations among Conjectures \ref{conjecture Chern number}, \ref{conjecture Chern character number} and $\text{Td}^{\frac12}[\cdot]>0$. It is not difficult from the identity
$$\log\frac{\sinh(x/2)}{x/2}=\sum_{k=1}^{\infty}\frac{B_{2k}}{(2k)!\cdot 2k}x^{2k}$$
to deduce that
\be\label{4}\text{Td}^{\frac12}\xlongequal{c_1=0}
\exp\big(-\sum_{k=1}^{\infty}\frac{B_{2k}}{4k}\cdot\text{ch}_{2k}\big).\ee
Note that the sign of $B_{2k}$ is $(-1)^{k-1}$ \big(see (\ref{Bernoulli-Riemann-Zeta})\big) and hence
$$\text{Conjecture \ref{conjecture Chern character number} $\xLongrightarrow{(\ref{4})}$ $\text{Td}^{\frac12}[\cdot]>0$}.$$

Nevertheless, our Theorem \ref{td1-2} implies that the signs of coefficients in front of (monomial) Chern numbers for $\text{Td}^{\frac12}[\cdot]$ can be both positive and negative.
So in general Conjecture \ref{conjecture Chern number} is \emph{not} able to deduce that $\text{Td}^{\frac12}[\cdot]>0$.

The next question is what the relations are between Conjecture \ref{conjecture Chern number} and Conjecture \ref{conjecture Chern character number}. Note that
\be\label{chern charater}C_{\lambda}[M]=\int_Me_{\lambda}(x_1,\ldots,x_n),\qquad
\text{Ch}_{\lambda}[M]=\frac{1}{\prod_{i=1}^{l(\lambda)}\lambda_i!}
\int_Mp_{\lambda}(x_1,\ldots,x_n)\ee
in the notation of (\ref{e-def}) and (\ref{p-def}), and for irreducible hyper-K\"{a}hler manifolds only those partitions whose parts are all even are involved. Simple example in the case of complex $4$-dimension reads
$$C_{(4)}=\frac12\text{Ch}_{(2,2)}-6\text{Ch}_{(4)},\qquad
\text{Ch}_{(4)}=\frac{1}{12}C_{(2,2)}-\frac{1}{6}C_{(4)}$$
and hence indicate that Conjectures \ref{conjecture Chern number} and \ref{conjecture Chern character number} are in general independent. Here we shall explain that Doubilet's results in \cite{Do72} give closed transformation formulas between Chern numbers and Chern character numbers of irreducible hyper-K\"{a}hler manifolds. Recall in \cite[Thm 3]{Do72} that
\begin{eqnarray}\label{transform e-p}
\left\{\begin{array}{ll}
e(\pi)=\sum_{\rho\leq\pi}\mu(\hat{0},\rho)\cdot p(\rho)\\
~\\
p(\pi)=\frac{1}{\mu(\hat{0},\pi)}\sum_{\rho\leq\pi}\mu(\rho,\pi)\cdot e(\rho),
\end{array} \right.
\end{eqnarray}
where the related notation can be found in Definition \ref{def of p,m,e,h} and Lemma \ref{Mobius lemma}.

Putting Lemma \ref{Mobius lemma}, Example \ref{Doubilet example}, (\ref{chern charater}) and (\ref{transform e-p}) together, we have
\begin{lemma}
Let $M$ be an irreducible hyper-K\"{a}hler manifold of complex dimension $2n$ and $2\nu$ an integer partition of weight $2n$. Fix a partition $\pi\in\Pi_{2n}$ such that $\lambda(\pi)=2\nu$. Then the Chern numbers and Chern character numbers of $M$ are related by
$$C_{2\nu}[M]=\frac{1}{(2\nu)!}\sum_{\overset{\rho\leq\pi}{\text{$|\rho_i|$ are even}}}\mu(\hat{0},\rho)\cdot\lambda(\rho)!\cdot\text{Ch}_{\lambda(\rho)}[M]$$
and
$$\text{Ch}_{2\nu}[M]=\frac{1}{(2\nu)!\cdot\mu(\hat{0},\pi)}
\sum_{\overset{\rho\leq\pi}{\text{$|\rho_i|$ are even}}}\mu(\rho,\pi)\cdot C_{\lambda(\rho)}[M],
$$
where the related notation can be found in Definition \ref{poset} and Lemma \ref{Mobius lemma}.
\end{lemma}

\section*{Acknowledgements}
The author sincerely thanks the referees for their careful reading and many helpful comments and remarks, which enhance the quality of this paper. One referee pointed out the third part in Remark \ref{SMZV remark}.
Another referee helped the author to make Section \ref{section-Doubilet formula} more readable. The author particularly appreciates both referees for pointing out a few subtle grammatical mistakes overlooked by the author for many years! This work was partially supported by the National
Natural Science Foundation of China (Grant No. 12371066).

\end{document}